\documentclass[12pt,draft]{article} 
\usepackage{subfiles}
\usepackage{amsmath,amssymb,amsthm}
\usepackage{geometry}
\usepackage{graphicx}
\DeclareGraphicsExtensions{.pdf,.jpeg,.png}
\usepackage{comment}
\usepackage{appendix}
\usepackage{csvsimple,longtable,supertabular}
\usepackage{multicol}
\usepackage[mathlines]{lineno}
\usepackage{indentfirst}
\usepackage{float}
\usepackage{xparse}
\usepackage[shortlabels]{enumitem}
    \NewDocumentCommand{\ceil}{s O{} m}{
  \IfBooleanTF{#1}
    {$\left\lceil#3\right\rceil$}
    {#2\lceil#3#2\rceil} 
}
\usepackage{hyperref}

\usepackage[usenames,dvipsnames]{xcolor}
\definecolor{USTpurple}{RGB}{80,7,120}
\allowdisplaybreaks

\newtheorem{thm}{Theorem}[section]

\newtheorem{prop}[thm]{Proposition}

\theoremstyle{definition}

\newtheorem{defn}[thm]{Definition}

\theoremstyle{remark}

\usepackage{tikz}
\usepackage{float}
\usepackage[nomessages]{fp}

\newcommand{\vertexdef}{\tikzstyle{vertex}=[draw,circle,fill=white,minimum size=15pt,inner sep=0pt]
			\tikzstyle{cv}=[white,draw,circle,fill=blue,minimum size = 15pt, inner sep=0pt]}

\begin{document}

\title{\bf{Zero Forcing of Generalized Hierarchical Products of Graphs}}
\author{Heather LeClair\thanks{Center for Applied Mathematics, University of St. Thomas, St. Paul,  MN   55105, USA}
    \and Tim Spilde\thanks{Center for Applied Mathematics, University of St. Thomas, St. Paul,  MN   55105, USA}
    \and Sarah Anderson\thanks{Department of Mathematics, Mail \#201, University of St. Thomas, 2115 Summit Avenue, St. Paul, MN 55015, USA (ande1298@stthomas.edu)}  
    \and Brenda Kroschel\thanks{Department of Mathematics, Mail \#OSS 201, University of St. Thomas, 2115 Summit Avenue, St. Paul, MN  55105, USA (bkkroschel@stthomas.edu)}}
\date{\today}

\maketitle

\begin{abstract}
   Zero forcing is a graph propagation process for which vertices fill-in (or propagate information to) neighbor vertices if all neighbors except for one, are filled. The zero-forcing number is the smallest number of vertices that must be filled to begin the process so that the entire graph or network becomes filled. In this paper, bounds are provided on the zero forcing number of generalized hierarchical products.
\end{abstract}

\section{Introduction}

A \textit{graph}, denoted $G=(V,E)$, consists of a set of vertices, $V(G) = \{v_1, v_2,...v_n\}$, and a set of edges, $E(G) = \{v_iv_j|$ $i,j \in \{1,...,n\}\}$. This paper will consider \textit{simple, undirected graphs}. A \textit{simple graph} has no more than one edge between two vertices and a single vertex cannot have an edge with itself (e.g. a loop). An \textit{undirected graph} has edges without orientation. The edges of the graph are expressed as a set of unordered pairs. The notation $v_iv_j$ denotes that there exists an edge between the vertices $v_i$ and $v_j$. 

\begin{figure}[H]\label{fig:basicgraph}
    \centering
\begin{tikzpicture}
    \vertexdef
    \node[vertex] (1) at (-2,1.154){1};
    \node[vertex] (2) at (-2,-1.154){2}; 
5

    \node[vertex] (3) at (2,1.154){3};
    \node[vertex] (4) at (2,-1.154){4};
    \draw (1) -- (2);
    \draw (2) -- (3);
    \draw (2) -- (4);
    \draw (3) -- (4);
\end{tikzpicture}
    \caption {$G =(\{1, 2, 3, 4\}$, $\{12$, $23$, $24$, $34$\})}
\end{figure}
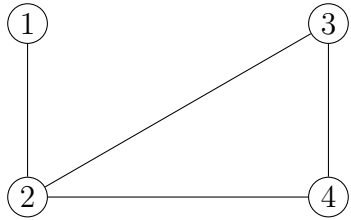

Two vertices that share an edge are considered \textit{neighbors} or \textit{adjacent}. The \textit{degree} of a vertex $v$ is determined by the number of adjacent vertices denoted $deg(v)$. 

\textit{Zero forcing} of a graph concerns the passing of information between neighboring vertices. This process follows a \textit{color change rule} which states that if a filled vertex, $v_n$, has only one neighbor, $v_m$,  that is unfilled, then $v_n$ \textit{forces} $v_m$ to be filled. When all of the vertices are filled, then the graph is \textit{color derived}. A \textit{zero forcing set} is any combination of vertices that are initially filled that have the potential to force the entire graph to be filled. A set with the least number of initial filled vertices is called a \textit{minimal zero forcing set}. The \textit{zero forcing number} of the graph, $Z(G)$, is the cardinality of a minimal zero forcing set. Let $S$ be a zero forcing set of the graph $G$. The \textit{propagation time}, denoted $pt(G,S)$ is the number of iterations needed for $S$ to force $G$ to be color-derived. 

\begin{center}
    \begin{tikzpicture}
        \vertexdef
        \node[cv] (1) at (-2,1.154){1};
        \node[vertex] (2) at (-2,-1.154){2};
        \node[cv] (3) at (2,1.154){3};
        \node[vertex] (4) at (2,-1.154){4};
        \draw (1) -- (2);
        \draw (2) -- (3);
        \draw (2) -- (4);
        \draw (3) -- (4);
    \end{tikzpicture}
    
\vspace{1cm}

\begin{tikzpicture}
        \vertexdef
        \node[cv] (1) at (-2,1.154){1};
        \node[cv] (2) at (-2,-1.154){2};
        \node[cv] (3) at (2,1.154){3};
        \node[vertex] (4) at (2,-1.154){4};
        \draw (1) -- (2);
        \draw (2) -- (3);
        \draw (2) -- (4);
        \draw (3) -- (4);
    \end{tikzpicture}
    
\vspace{1cm}

\begin{figure} [H]
    \centering
\begin{tikzpicture}
        \vertexdef
        \node[cv] (1) at (-2,1.154){1};
        \node[cv] (2) at (-2,-1.154){2};
        \node[cv] (3) at (2,1.154){3};
        \node[cv] (4) at (2,-1.154){4};
        \draw (1) -- (2);
        \draw (2) -- (3);
        \draw (2) -- (4);
        \draw (3) -- (4);
    \end{tikzpicture}

\caption{$Z(G)=2$; $pt(G,S)=2$}
\label{intrograph}

\end{figure}
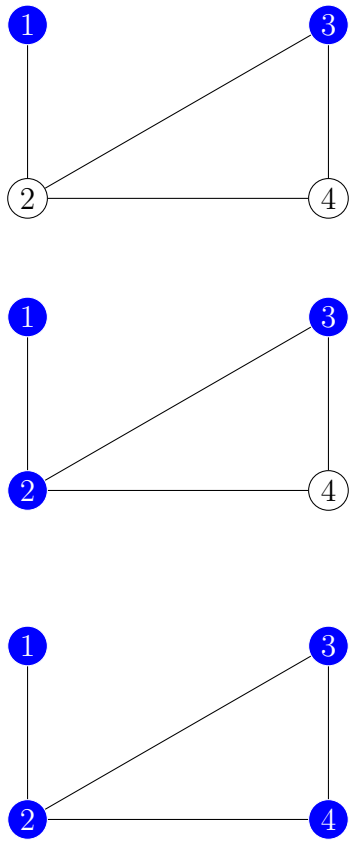
\end{center}

Consider the graph from Figure \ref{intrograph}. Assume one starts with the zero forcing set $S=\{1,2\}$. The vertex $1$ has only one unfilled 
neighbor, vertex $2$, so it forces vertex $2$ to be filled. The vertex $3$ has two neighbors, vertex $2$ and vertex $4$. The only unfilled neighbor of vertex $3$ is vertex $4$. Thus, vertex $3$ forces vertex $4$ to be filled, and the entire graph is filled. 

In this paper, bounds are given on the zero forcing number of generalized hierarchical products. In Section \ref{classes}, special classes of graphs are defined as well as their zero forcing number. Section \ref{products} provides an introduction to the Cartesian product and generalized hierarchical product. Section \ref{matrices} relates zero forcing to the minimum rank problem. In Section \ref{specialproducts}, bounds are given on the zero forcing number of the generalized hierarchical product for special classes of graphs, and in Section \ref{generalboundsec}, an upper bound is given on the zero forcing number of the generalized hierarchical product for any two arbitrary graphs. 

\section{Special Classes of Graphs}
\label{classes}
There are three special classes of graphs that are relevant to this paper: paths, cycles, and complete graphs. %These special graphs will be used as factors in the generalized hierarchical products of graphs, which are the main focus of this paper.

\begin{defn} A {path}, denoted $P_n$, consists of $n$ vertices $\{1,2,...,n\}$ and edges $\{12, 23,...,$ $(n-1)n\}$.

\end{defn}

\begin{figure} [H]
    \centering
    \begin{tikzpicture}
        \vertexdef
            \node[cv] (1) at (-2,0) {1};
            \node[vertex] (2) at (-1,0) {2};
            \node[vertex] (3) at (0,0) {3};
            \node[vertex] (4) at (1,0) {4};
            \node[vertex] (5) at (2,0) {5};
            \draw (1) -- (2);
            \draw (2) -- (3);
            \draw (3) -- (4);
            \draw (4) -- (5);
    \end{tikzpicture}
    \caption{$G = P_5$}
\end{figure}

For a path, $Z(P_n)=1$ with a minimal zero forcing set $S=\{1\}$ or $\{n\}$. Having one unfilled neighbor, the vertex in $S$ can fill the remaining vertices one at a time in order. The propagation time is $pt(P_n,S)=n-1$.

\begin{defn} A {cycle}, denoted $C_n$, for $n\geq3$, consists of $n$ vertices $\{1,2,...,n\}$ and edges $\{12, 23,..., (n-1)n,n1\}$.

\end{defn}

\begin{figure} [H]
    \centering
    \begin{tikzpicture}
        \vertexdef
            \node[cv] (1) at (0,2) {1};
            \node[cv] (2) at (1.732,1) {2};
            \node[vertex] (3) at (0.866,-0.802) {3};
            \node[vertex] (4) at (-0.866,-0.802) {4};
            \node[vertex] (5) at (-1.732,1) {5};
            \draw (1) -- (2);
            \draw (2) -- (3);
            \draw (3) -- (4);
            \draw (4) -- (5);
            \draw (5) -- (1);
    \end{tikzpicture}
    \caption{$G = C_5$}
\end{figure}

For a cycle, $Z(C_n)=2$, and any two adjacent vertices form a minimal zero forcing set, $S$. Since each vertex in $S$ has one unfilled neighbor, either filled vertex can filled similarly its unfilled neighbor. Then the remaining vertices are filled around the cycle. The  propagation time is $pt(C_n,S)= \ceil[\big]{\frac{n-2}{2}}$.
%The numbering of vertices is arbitrary, meaning that the vertices can be numbered in any manner as long as each vertex has degree two. Thus, a cycle is isomorphic because it can be rotated and direction and is still effectively the same graph.  

\begin{defn} A {complete graph}, denoted $K_n$, consists of $n$ vertices $\{1,2,...,n\}$ and contains exactly one edge between each pair of distinct vertices.

\end{defn}

\begin{figure} [H]
    \centering
    \begin{tikzpicture}
        \vertexdef
            \node[cv] (1) at (0,2) {1};
            \node[cv] (2) at (1.732,1) {2};
            \node[cv] (3) at (0.866,-0.802) {3};
            \node[cv] (4) at (-0.866,-0.802) {4};
            \node[vertex] (5) at (-1.732,1) {5};
            \draw (1) -- (2);
            \draw (1) -- (3);
            \draw (1) -- (4);
            \draw (1) -- (5);
            \draw (2) -- (3);
            \draw (2) -- (4);
            \draw (2) -- (5);
            \draw (3) -- (4);
            \draw (3) -- (5);
            \draw (4) -- (5);
    \end{tikzpicture}
    \caption{$G = K_5$}
\end{figure}

 For a complete graph, $Z(K_n)=n-1$. In order for a filled vertex to fill an unfilled neighbor, $n-2$ of its neighbors must be filled. Thus, a set consisting of all but one of the vertices forms a minimal zero forcing set, $S$. The propagation time is $pt(K_n,S)=1$.

\section{Graph Products}
\label{products}
The two cases of graph products that this paper will consider are the \textit{Cartesian product} and the \textit{generalized hierarchical product}.

\subsection{Cartesian Product}
We will first consider the Cartesian product.

\begin{defn}
Given graphs $W$ and $H$, the Cartesian product, denoted $G = W \Box H$, is the graph with vertex set is $V(W\Box H) = V(W) \times V(H)$ where two vertices $(x_1,y_1)$ and $(x_2,y_2)$ are adjacent in $W
\Box H$ if either $x_1=x_2$ and $y_1y_2$ is an edge in $H$, or $y_1=y_2$
and $x_1x_2$ is an edge in $W$. 
\end{defn}

Essentially, the vertex set of $W\square H$ is the Cartesian product of the vertex sets of $W$ and $H$. The edge set is determined by factors the Cartesian product. That is, the edges of $W$ are copied along the vertices of $H$. The vertices in these copies of $G$ will be referred to as rows. An arbitrary \textit{row} $j=\{(x,j) \; | \; x \in V(W)\}$ for $j=1,...,h$. Further, the edges of $H$ are copied along the vertices of $W$. The vertices in these copies of $H$ will be referred to as columns. An arbitrary \textit{column} $i=\{(i,y) \; | \; y \in V(H)\}$ for $i=1,...,w$.

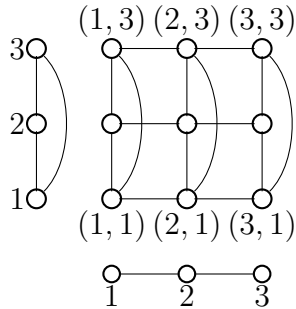
\begin{figure} [H]
    \centering
    \begin{tikzpicture}
        \vertexdef
         \filldraw[color=black, fill=white, thick](1,0) circle (0.11) node[anchor=north] {$1$}; %1
         \filldraw[color=black, fill=white, thick](2,0) circle (0.11) node[anchor=north] {$2$};  %2
         \filldraw[color=black, fill=white, thick](3,0) circle (0.11) node[anchor=north] {$3$};%3
         
         \draw (1.1,0) -- (1.9,0);
         \draw (2.1,0) -- (2.9,0);
        
             \filldraw[color=black, fill=white, thick](0,1) circle (0.125) node[anchor=east] {$1$};%4
             \filldraw[color=black, fill=white, thick](0,2) circle (0.125) node[anchor=east] {$2$};%5
             \filldraw[color=black, fill=white, thick](0,3) circle (0.125) node[anchor=east] {$3$};%6
           
           \draw (0,1.1) -- (0,1.9);
           \draw (0,2.1) -- (0,2.9);
           
            \filldraw[color=black, fill=white, thick](1,1) circle (0.125) node[anchor=north] {$(1,1)$};%7
            \filldraw[color=black, fill=white, thick](2,1) circle (0.125) node[anchor=north] {$(2,1)$};%8
            \filldraw[color=black, fill=white, thick](3,1) circle (0.125) node[anchor=north] {$(3,1)$};%9
            
            \draw (1.1,1) -- (1.9,1);
            \draw (2.1,1) -- (2.9,1);
            \draw (1,1.1) -- (1,1.9);
           \draw (1,2.1) -- (1,2.9);
            
             \filldraw[color=black, fill=white, thick](1,2) circle (0.125) {};%10
             \filldraw[color=black, fill=white, thick](2,2) circle (0.125){};%11
            \filldraw[color=black, fill=white, thick](3,2) circle (0.125){};%12
            
            \draw (1.1,2) -- (1.9,2);
            \draw (2.1,2) -- (2.9,2);
            \draw (2,1.1) -- (2,1.9);
           \draw (2,2.1) -- (2,2.9);
            
          \filldraw[color=black, fill=white, thick](1,3) circle (0.125) node[anchor=south] {$(1,3)$};%13
          \filldraw[color=black, fill=white, thick](2,3) circle (0.125) node[anchor=south] {$(2,3)$};%14
          \filldraw[color=black, fill=white, thick](3,3) circle (0.125) node[anchor=south] {$(3,3)$};%15
           
           \draw (1.1,3) -- (1.9,3);
            \draw (2.1,3) -- (2.9,3);
            \draw (3,1.1) -- (3,1.9);
           \draw (3,2.1) -- (3,2.9);
          
            \draw (0.1,1.1)..controls (0.5,1.5) and (0.5,2.5)..(0.1,2.9);
            \draw (1.1,1.1)..controls(1.5,1.5)and (1.5,2.5)..(1.1,2.9);
            \draw (2.1,1.1)..controls (2.5,1.5) and (2.5,2.5)..(2.1,2.9);
            \draw (3.1,1.1)..controls (3.5,1.5) and (3.5,2.5)..(3.1,2.9);

    \end{tikzpicture}
    \caption{$G = P_3\square C_3$}
    \label{cartesian}
\end{figure}

\noindent
Figure \ref{cartesian} is an example of a Cartesian product of a path and cycle: $P_3\square C_3$. Notice that there are three copies of $P_3$ across the rows as well as three copies of $C_3$ along the columns.

The zero forcing number of the Cartesian product has been studied previously. In \cite{AIM}, an upper bound on the zero forcing number of the Cartesian product of any two arbitrary graphs was given as well as lower bounds on the zero forcing number of the Cartesian products with particular factors. In addition, \cite{proptime} studied the propagation time of various Cartesian products.

\subsection{Generalized Hierarchical Products}
The generalized hierarchical product of graphs was first introduced by Barri\`{e}re et al. in 2009 \cite{hier}. The generalized hierarchical product, which is a generalization of the the Cartesian product, can be used to can be used to model scale-free networks \cite{scalefree, newman}. In this paper, the ``generalized hierarchical product'' will be referred to as the ``hierarchical product.''

\begin{defn}
Given graphs $W$ and $H$ and vertex subset $U \subseteq V(W)$, the hierarchical product, denoted $G = W(U) \sqcap H$, is the graph with vertex set $V(W) \times V(H)$ where two vertices $(x_1,y_1)$ and $(x_2,y_2)$ are adjacent in $G = W(U) \sqcap H$ if either $x_1 \in U$, $x_1=x_2$ and $y_1y_2$ is an edge in $H$, or $y_1=y_2$
and $x_1x_2$ is an edge in $W$. The root set, $U$, determines the edge set for the columns of the graph. The  vertices in the root set are a subset of the vertices in $W$ and dictate the columns that receive the edge pattern of $H$. In general, the root set is denoted $U=\{i_1,i_2,...,i_k\}$ such that $i_1 < i_2<...<i_k$.
\end{defn}

A Cartesian product is an special case of a hierarchical product. A Cartesian product is a hierarchical product that has all of the vertices of $G$ included in the root set; that is, $U=V(W)$. 
Figure \ref{pathandcycle} depicts a hierarchical graph product of a path and cycle with a root set $U=\{1\}$. Notice that there are three copies of $P_3$ across the rows as well as one copy of $C_3$ along the column 1.

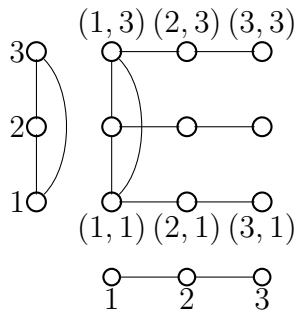
\begin{figure} [H]
    \centering
    \begin{tikzpicture}
        \vertexdef
        
            \filldraw[color=black, fill=white, thick](1,0) circle (0.11) node[anchor=north] {$1$}; %1
         \filldraw[color=black, fill=white, thick](2,0) circle (0.11) node[anchor=north] {$2$};  %2
         \filldraw[color=black, fill=white, thick](3,0) circle (0.11) node[anchor=north] {$3$};%3
         
         \draw (1.1,0) -- (1.9,0);
         \draw (2.1,0) -- (2.9,0);
        
             \filldraw[color=black, fill=white, thick](0,1) circle (0.125) node[anchor=east] {$1$};%4
             \filldraw[color=black, fill=white, thick](0,2) circle (0.125) node[anchor=east] {$2$};%5
             \filldraw[color=black, fill=white, thick](0,3) circle (0.125) node[anchor=east] {$3$};%6
           
           \draw (0,1.1) -- (0,1.9);
           \draw (0,2.1) -- (0,2.9);
           
            \filldraw[color=black, fill=white, thick](1,1) circle (0.125) node[anchor=north] {$(1,1)$};%7
            \filldraw[color=black, fill=white, thick](2,1) circle (0.125) node[anchor=north] {$(2,1)$};%8
            \filldraw[color=black, fill=white, thick](3,1) circle (0.125) node[anchor=north] {$(3,1)$};%9
            
            \draw (1.1,1) -- (1.9,1);
            \draw (2.1,1) -- (2.9,1);
            \draw (1,1.1) -- (1,1.9);
           \draw (1,2.1) -- (1,2.9);
            
             \filldraw[color=black, fill=white, thick](1,2) circle (0.125) {};%10
             \filldraw[color=black, fill=white, thick](2,2) circle (0.125){};%11
            \filldraw[color=black, fill=white, thick](3,2) circle (0.125) {};%12
            
            \draw (1.1,2) -- (1.9,2);
            \draw (2.1,2) -- (2.9,2);
            
          \filldraw[color=black, fill=white, thick](1,3) circle (0.125) node[anchor=south] {$(1,3)$};%13
          \filldraw[color=black, fill=white, thick](2,3) circle (0.125)node[anchor=south] {$(2,3)$};%14
          \filldraw[color=black, fill=white, thick](3,3) circle (0.125) node[anchor=south] {$(3,3)$};%15
           
           \draw (1.1,3) -- (1.9,3);
            \draw (2.1,3) -- (2.9,3);

            \draw (0.1,1.1)..controls (0.5,1.5) and (0.5,2.5)..(0.1,2.9);
            \draw (1.1,1.1)..controls(1.5,1.5)and (1.5,2.5)..(1.1,2.9);
            
    \end{tikzpicture}
    \caption{$G = P_3(U)\sqcap C_3,$ $U=\{1\} $}
    \label{pathandcycle}
\end{figure}

\begin{defn}
The simple graphs $G_1=(V_1,E_1)$ and $G_2=(V_2,E_2)$ are \textit{isomorphic} if there exists a one-to-one and onto function $f$ from $V_1$ to $V_2$ with the property that $a$ and $b$ are adjacent in $G_1$ if and only if $f(a)$ and $f(b)$ are adjacent in $G_2$, for all $a$ and $b$ in $V_1$.
\end{defn}

In other words, when two simple graphs are isomorphic, there is a one-to-one correspondence between vertices of the two graphs that preserves the associated relationship.

Some graph products have different root sets, but are isomorphic. In Figure \ref{isograph}, the root set of each graph has two adjacent elements. The graph on the left has $U=\{1,2\}$ while the graph on the right has $U=\{2,3\}$. The right graph is the same exact graph as the left but horizontally rotated. Thus, they have the same zero forcing number. 

%This applies when the first competent in the graph product is a cycle or complete graph.

\begin{figure} [H]
\begin{center}
    \begin{tikzpicture}
        \vertexdef
           
            \filldraw[color=black, fill=white, thick](1,1) circle (0.125) node[anchor=north] {} ;%7
            \filldraw[color=black, fill=white, thick](2,1) circle (0.125) node[anchor=north] {};%8
            \filldraw[color=black, fill=white, thick](3,1) circle (0.125) node[anchor=north] {};%9
            
            \draw (1.1,2.1)..controls (1.5,2.5) and (2.5,2.5)..(2.9,2.1);
            \draw (1.1,1) -- (1.9,1);
            \draw (2.1,1) -- (2.9,1);
            \draw (1,1.1) -- (1,1.9);
           \draw (1,2.1) -- (1,2.9);
           \draw (1.1,1.1)..controls (1.5,1.5) and (2.5,1.5)..(2.9,1.1);
            
             \filldraw[color=black, fill=white, thick](1,2) circle (0.125) {};%10
             \filldraw[color=black, fill=white, thick](2,2) circle (0.125) {};%11
            \filldraw[color=black, fill=white, thick](3,2) circle (0.125) {};%12
            \draw (1.1,3.1)..controls (1.5,3.5) and (2.5,3.5)..(2.9,3.1) {};
            \draw (1.1,2) -- (1.9,2);
            \draw (2.1,2) -- (2.9,2);
            \draw (2,1.1) -- (2,1.9);
           \draw (2,2.1) -- (2,2.9);
            
          \filldraw[color=black, fill=white, thick](1,3) circle (0.125) {};%13
          \filldraw[color=black, fill=white, thick](2,3) circle (0.125) {};%14
          \filldraw[color=black, fill=white, thick](3,3) circle (0.125) {};%15
           
           \draw (1.1,3) -- (1.9,3);
            \draw (2.1,3) -- (2.9,3);
          
            \draw (1.1,1.1)..controls(1.5,1.5)and (1.5,2.5)..(1.1,2.9);
            \draw (2.1,1.1)..controls (2.5,1.5) and (2.5,2.5)..(2.1,2.9);
  %%%%%%%%%%%%%%%%%%%%%%%%%%%%%%%%%%%%%%%%%%%%%%%%         

            \filldraw[color=black, fill=white, thick](5,1) circle (0.125) node[anchor=north] {};%7
            \filldraw[color=black, fill=white, thick](6,1) circle (0.125) node[anchor=north] {};%8
            \filldraw[color=black, fill=white, thick](7,1) circle (0.125) node[anchor=north] {};%9
            
            \draw (5.1,2.1)..controls (5.5,2.5) and (6.5,2.5)..(6.9,2.1);
            \draw (5.1,1) -- (5.9,1);
            \draw (6.1,1) -- (6.9,1);
            \draw (7,1.1) -- (7,1.9);
           \draw (7,2.1) -- (7,2.9);
            \draw (5.1,1.1)..controls (5.5,1.5) and (6.5,1.5)..(6.9,1.1);
            
             \filldraw[color=black, fill=white, thick](5,2) circle (0.125);%10
             \filldraw[color=black, fill=white, thick](6,2) circle (0.125);%11
            \filldraw[color=black, fill=white, thick](7,2) circle (0.125);%12
            \draw (5.1,3.1)..controls (5.5,3.5) and (6.5,3.5)..(6.9,3.1);
            \draw (5.1,2) -- (5.9,2);
            \draw (6.1,2) -- (6.9,2);
            \draw (6,1.1) -- (6,1.9);
           \draw (6,2.1) -- (6,2.9);
            
          \filldraw[color=black, fill=white, thick](5,3) circle (0.125) {};%13
          \filldraw[color=black, fill=white, thick](6,3) circle (0.125) {};%14
          \filldraw[color=black, fill=white, thick](7,3) circle (0.125) {};%15
           
           \draw (5.1,3) -- (5.9,3);
            \draw (6.1,3) -- (6.9,3);

            \draw (7.1,1.1)..controls(7.5,1.5)and (7.5,2.5)..(7.1,2.9);
            \draw (6.1,1.1)..controls (6.5,1.5) and (6.5,2.5)..(6.1,2.9);
            
    \end{tikzpicture}
    \caption{$G = C_3(U)\sqcap C_3$ with $U=\{1,2\}$ and $H= C_3(U)\sqcap C_3$ with $U=\{2,3\}$}
    \label{isograph}
\end{center}
\end{figure}
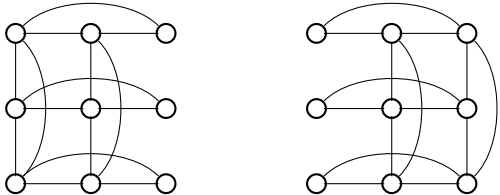

\section{Associated Matrices}
\label{matrices}
\subsection{Constructing an Associated Matrix}
For a simple graph $G=(V,E)$ with vertices $v_1,v_2,...,v_n$ define an  \textit{associated matrix}, $A(G)$, as the $n \times n$
matrix with the $(i,j)^{th}$ entry nonzero precisely when $v_i$ and $v_j$ are adjacent, and $0$ as when they are not adjacent. The diagonal entries that represent a single vertex are free to be chosen as zero or nonzero. The family of associated matrices can be represented by a matrix in which the nonzero entries are represented by a ``*" and the diagonal entries, which are free to be zero or nonzero, are represented by ``?".  Figure \ref{assocmatrix} is an example of a graph with its  associated zero-nonzero pattern matrix on the right.

%\begin{center}
%    \begin{equation*}
%        a_{ij}
%        \begin{cases}
%            \neq 0 & \text{if $v_iv_j$ is an edge of G,}\\
%            ? & \text{if $i=j$,}\\
%            0 & \text{otherwise.}
%        \end{cases}
 %   \end{equation*}
%\end{center}

\begin{figure}[H]
\begin{multicols}{2}
        \begin{center}
        \begin{tikzpicture}
            \vertexdef
            \node[vertex] (1) at (-2,1.154){1};
            \node[vertex] (2) at (-2,-1.154){2};
            \node[vertex] (3) at (2,1.154){3};
            \node[vertex] (4) at (2,-1.154){4};
            \draw  (1) -- (2);
            \draw  (4) -- (2);
            \draw  (3) -- (2);
            \draw  (3) -- (4);
        \end{tikzpicture}
        \end{center}

        \begin{center}
        $$\begin{pmatrix}
         ?& * & 0 & 0 \\
         * & ? & * & * \\
         0 & * & ? & * \\
         0 & * & * & ? \\
        \end{pmatrix}$$
        \end{center}
\end{multicols}
\caption{$G =(\{1, 2, 3, 4\}$, $\{12$, $23$, $24$, $34$\})}
\label{assocmatrix}
\end{figure}
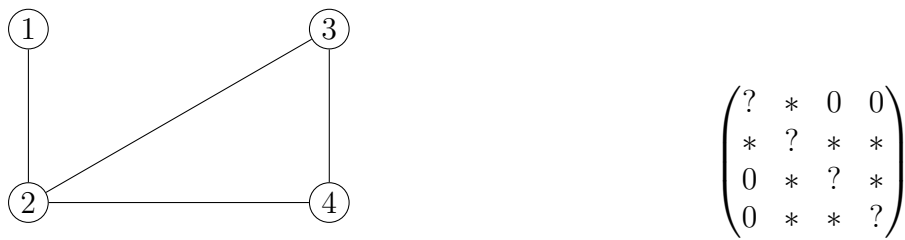

In addition, given a $n \times n$ matrix $A = (a_{ij})$, then the \textit{graph of $A$}, $G(A)$, is the graph with vertices $\{1, \ldots, n\}$ in which there exists an edge $ij$ for $1 \leq i < j \leq n$ if and only if $a_{ij} \neq 0$.

 Define the \textit{set of symmetric matrices} of graph $G$ as $S(G) = \{A \in S_n \; | \; G(A) = G\}$ in which $S_n$ is the set of all $n \times n$ symmetric matrices.  The \textit{minimum rank} of a graph $G$ is defined to be $mr(G)=\min\{rank(A) \; | \; A \in S(G)\}$. The minimum rank problem asks what is the minimum rank over all symmetric matrices $A(G)\in S(G)$ that have the associated zero, nonzero pattern.  In \cite{AIM}, it was proven that the minimum rank of a graph provides a lower bound on $Z(G)$.

%A given associated matrix has an infinite number of possible combination of numbers that could be filled in for the $*$ and $?$ entries. Our goal is to place numbers in these entries that will give us the proper rank of the matrix in order to determine $Z(G)$. 

%``Zero forcing sets and the minimum rank of graphs'' is a paper published in 2007 that first proved that the minimum rank of an associated matrix gives a lower bound on $Z(G)$.

\begin{prop} \cite{AIM}
Let $G=(V,E)$ be a graph, let $Z \subseteq V$ be a zero forcing set, and let $n$ be the number of columns in $A(G)$. Then \\ $n-mr(G) \leq |Z|$.
\label{minrank}
\end{prop}

\noindent Thus, if $Z(G)$ a minimal zero forcing set of a graph $G$, then Proposition \ref{minrank} states \\ $n - mr(G) \leq Z(G)$.

\section{Zero Forcing of Hierarchical Products}
\label{specialproducts}
In this section, we consider the zero forcing number of hierarchical products of paths, cycles, and complete graphs. The vertices of these graphs will be labeled as in Section \ref{classes}. Figure \ref{pathbypath} demonstrates this labeling for $G = P_4(U)\sqcap P_3$ with $U=\{1,3\}$. 

\subsection{Path by Path}
The first hierarchical product considered is the product of two paths, $P_w(U) \sqcap P_h$, with various root sets, $U$. 

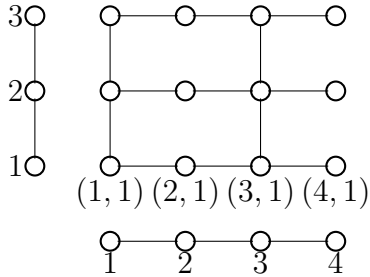
\begin{figure} [H]
\begin{center}
    \begin{tikzpicture}
        \vertexdef
        
         \filldraw[color=black, fill=white, thick](1,0) circle (0.125) node[anchor=north] {$1$}; 
         \filldraw[color=black, fill=white, thick](2,0) circle (0.125) node[anchor=north] {$2$};  
         \filldraw[color=black, fill=white, thick](3,0) circle (0.125) node[anchor=north] {$3$};
        \filldraw[color=black, fill=white, thick](4,0) circle (0.125) node[anchor=north] {$4$};
         
         \draw (1.1,0) -- (1.9,0);
         \draw (2.1,0) -- (2.9,0);
         \draw (3.1,0) -- (3.9,0);
        
             \filldraw[color=black, fill=white, thick](0,1) circle (0.125) node[anchor=east] {$1$};
             \filldraw[color=black, fill=white, thick](0,2) circle (0.125) node[anchor=east] {$2$};
             \filldraw[color=black, fill=white, thick](0,3) circle (0.125) node[anchor=east] {$3$};
           
           \draw (0,1.1) -- (0,1.9);
           \draw (0,2.1) -- (0,2.9);
           
            \filldraw[color=black, fill=white, thick](1,1) circle (0.125) node[anchor=north] {$(1,1)$};
            \filldraw[color=black, fill=white, thick](2,1) circle (0.125) node[anchor=north] {$(2,1)$};
            \filldraw[color=black, fill=white, thick](3,1) circle (0.125) node[anchor=north] {$(3,1)$};
            \filldraw[color=black, fill=white, thick](4,1) circle (0.125) node[anchor=north] {$(4,1)$};
            \filldraw[color=black, fill=white, thick](4,2) circle (0.125);
            \filldraw[color=black, fill=white, thick](4,3) circle (0.125);
            
            \draw (1.1,1) -- (1.9,1);
            \draw (2.1,1) -- (2.9,1);
            \draw (1,1.1) -- (1,1.9);
           \draw (1,2.1) -- (1,2.9);
            
             \filldraw[color=black, fill=white, thick](1,2) circle (0.125);
             \filldraw[color=black, fill=white, thick](2,2) circle (0.125);
            \filldraw[color=black, fill=white, thick](3,2) circle (0.125);
            
            \draw (1.1,2) -- (1.9,2);
            \draw (2.1,2) -- (2.9,2);
            \draw (3,1.1) -- (3,1.9);
           \draw (3,2.1) -- (3,2.9);
            
          \filldraw[color=black, fill=white, thick](1,3) circle (0.125);
          \filldraw[color=black, fill=white, thick](2,3) circle (0.125);
          \filldraw[color=black, fill=white, thick](3,3) circle (0.125);
           
           \draw (1.1,3) -- (1.9,3);
            \draw (2.1,3) -- (2.9,3);
            \draw (3.1,1) -- (3.9,1);
            \draw (3.1,2) -- (3.9,2);
            \draw (3.1,3) -- (3.9,3);

    \end{tikzpicture}
    \caption{$G = P_4(U) \sqcap P_3$ with $U=\{1,3\}$}
    \label{pathbypath}
\end{center}
\end{figure}

\subsubsection{Root Set of Cardinality 1}
Consider $P_w(U) \sqcap P_h$ with $|U|=1$.

\begin{thm}
\label{pathsize1}
Let $G=P_w(U) \sqcap P_h.$ 
\begin{enumerate}
    \item If $U=\{1\}$ or $U=\{w\}$, then $Z(G) \leq \ceil[\big]{\frac{h}{2}}$.
    \item If $U=\{i\}$ and $i \neq 1,w$, then $Z(G) \leq h$.
\end{enumerate}
\end{thm}

\begin{proof} 
\textbf{Case 1:} \\
Assume $h$ is even and $U=\{w\}$. Note when $U=\{1\}$, these graphs are isomorphic, and the results will follow similarly. Let
             \begin{center}
                  $S=\{(1,2),(1,3),(1,5),(1,7)...,(1,h-1)\}$.
             \end{center}
That is, $S$ consists of the vertices in the odd rows of column $1$ starting with row $3$ as well as the vertex $(1,2)$. We will show that $S$ is a zero forcing set. Note that all of these vertices have degree 1 and the remaining vertices in these rows are paths, so they are filled in order from left to right across each row. 

\begin{figure}[H]
\centering
    \begin{tikzpicture}
             \vertexdef
             
             \filldraw[color=black, fill=white, thick](-2,-1.5) circle (0.1) node[anchor=north] {$1$} node[anchor=east] {$1$};
           \filldraw[color=black, fill=white, thick](-1,-1.5) circle (0.1) node[anchor=north] {$2$};
           \filldraw[color=black, fill=white, thick](0,-1.5) circle (0.1) node[anchor=north] {$w-1$};
           \filldraw[color=black, fill=white, thick](1,-1.5) circle (0.1) node[anchor=north] {$w$};
           
            \draw (-1.9,-1.5) -- (-1.1,-1.5); \draw (0.1,-1.5) -- (0.9,-1.5); 
            
            \filldraw[black] (-1.5,-0.3) circle (0.5pt); \filldraw[black] (-1.5,0) circle (0.5pt); \filldraw[black] (-1.5,0.3) circle (0.5pt);
           \filldraw[black] (0.5,-0.3) circle (0.5pt); \filldraw[black] (0.5,0) circle (0.5pt); \filldraw[black] (0.5,0.3) circle (0.5pt);
           \filldraw[black] (-0.5,-0.3) circle (0.5pt); \filldraw[black] (-0.5,0) circle (0.5pt); \filldraw[black] (-0.5,0.3) circle (0.5pt);
           
           \filldraw[black] (-0.7,1.5) circle (0.5pt); \filldraw[black] (-0.5,1.5) circle (0.5pt); \filldraw[black] (-0.3,1.5) circle (0.5pt);
           \filldraw[black] (-0.7,1) circle (0.5pt); \filldraw[black] (-0.5,1) circle (0.5pt); \filldraw[black] (-0.3,1) circle (0.5pt);
           \filldraw[black] (-0.7,0.5) circle (0.5pt); \filldraw[black] (-0.5,0.5) circle (0.5pt); \filldraw[black] (-0.3,0.5) circle (0.5pt);
           \filldraw[black] (-0.7,-0.5) circle (0.5pt); \filldraw[black] (-0.5,-0.5) circle (0.5pt); \filldraw[black] (-0.3,-0.5) circle (0.5pt);
           \filldraw[black] (-0.7,-1) circle (0.5pt); \filldraw[black] (-0.5,-1) circle (0.5pt); \filldraw[black] (-0.3,-1) circle (0.5pt);
           \filldraw[black] (-0.7,-1.5) circle (0.5pt); \filldraw[black] (-0.5,-1.5) circle (0.5pt); \filldraw[black] (-0.3,-1.5) circle (0.5pt);
            
           \filldraw[color=black, fill=blue, thick](-2,-0.5) circle (0.1) node[anchor=east] {$3$};
           \filldraw[color=black, fill=blue, thick](-1,-0.5) circle (0.1);
           \filldraw[color=black, fill=blue, thick](0,-0.5) circle (0.1);
           \filldraw[color=black, fill=blue, thick](1,-0.5) circle (0.1);
           
            \draw (-1.9,-0.5) -- (-1.1,-0.5); \draw (0.1,-0.5) -- (0.9,-0.5); 
            
           \filldraw[color=black, fill=blue, thick](-2,-1) circle (0.1) node[anchor=east] {$2$};
           \filldraw[color=black, fill=blue, thick](-1,-1) circle (0.1);
           \filldraw[color=black, fill=blue, thick](0,-1) circle (0.1);
           \filldraw[color=black, fill=blue, thick](1,-1) circle (0.1);
           
            \draw (-1.9,-1) -- (-1.1,-1); \draw (0.1,-1) -- (0.9,-1);
            
            \filldraw[color=black, fill=blue, thick](-2,1) circle (0.1) node[anchor=east] {$h-1$};
           \filldraw[color=black, fill=blue, thick](-1,1) circle (0.1);
           \filldraw[color=black, fill=blue, thick](0,1) circle (0.1);
           \filldraw[color=black, fill=blue, thick](1,1) circle (0.1);
           
            \draw (-1.9,1) -- (-1.1,1); \draw (0.1,1) -- (0.9,1); 
            
            \filldraw[color=black, fill=white, thick](-2,0.5) circle (0.1) node[anchor=east] {$h-2$};
           \filldraw[color=black, fill=white, thick](-1,0.5) circle (0.1);
           \filldraw[color=black, fill=white, thick](0,0.5) circle (0.1);
           \filldraw[color=black, fill=white, thick](1,0.5) circle (0.1);
           
            \draw (-1.9,0.5) -- (-1.1,0.5); \draw (0.1,0.5) -- (0.9,0.5);
            
            \filldraw[color=black, fill=white, thick](-2,1.5) circle (0.1) node[anchor=east] {$h$};
           \filldraw[color=black, fill=white, thick](-1,1.5) circle (0.1);
           \filldraw[color=black, fill=white, thick](0,1.5) circle (0.1);
           \filldraw[color=black, fill=white, thick](1,1.5) circle (0.1);
           
            \draw (-1.9,1.5) -- (-1.1,1.5); \draw (0.1,1.5) -- (0.9,1.5); 
            
            \draw (1,-1.4) -- (1,-1.1);\draw (1,-0.9) -- (1,-0.6);\draw (1,-0.4) -- (1,-0.2);\draw (1,0.6) -- (1,0.9);\draw (1,1.1) -- (1,1.4);

    \end{tikzpicture}
    \caption{$G=P_w(U)\sqcap P_h$ with $U = \{w\}$}
    \label{pathwfilling}

\end{figure}
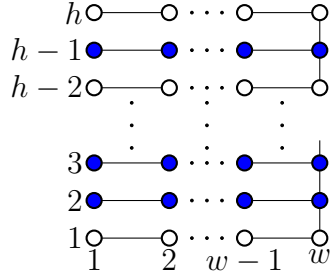

That is, vertices $(w,2)$ and $(w,3)$ are adjacent filled vertices each with degree 3. At this point, as depicted in Figure \ref{pathwfilling}, each row with a vertex in $S$ is filled. That is, row $i$ when $i$ is odd, and row $2$ are entirely filled. Filling now proceeds vertically to the remaining vertices in column $w$. Note that vertices $(w,2)$ and $(w,3)$ are adjacent filled vertices each with degree $3$. Since row $2$ and row $3$ are filled, these vertices each have one unfilled neighbor, $(w,1)$ and $(w,4)$, respectively. Thus, $(w,2)$ and $(w,3)$ force $(w,1)$ and $(w,4)$ to be filled simultaneously. 
Since the remaining vertices in rows $1$ and $4$ form paths, they are filled across the rows in reverse order from right to left. 
Recall that all of the vertices in odd rows are filled, so now, $(w,5)$ has one unfilled neighbor, $(w,6)$, so it is filled. Following the same reasoning as row $4$ and row $6$ is filled. This process continues until $(w,h-1)$ fills $(w,h)$. Note, $deg(w,h)=2$, with $(w,h-1)$ as its filled neighbor. Therefore, $(w,h)$ fills $(w-1,h)$. Then the rest of row $h$ is filled because the vertices form a path. Thus, the whole graph is filled and $Z(G) \leq {\frac{h}{2}}$.

Assume $h$ is odd and $U=\{w\}$. Let
        \begin{center}
            $S=\{(1,1), (1,2), (1,4),(1,6)...,(1,h-1)\}$.
        \end{center}
That is, $S$ consists of the vertices in the even rows of column $1$ as well as the vertex $(1,1)$. We will show that $S$ is a zero forcing set. Note that the vertices in $S$ all have degree 1 and the remaining vertices in these rows are paths, so they are filled across each row in order from left to right. Now, all of the rows that have a vertex in $S$ are filled. Note, $deg(w,2)=3$ with its only unfilled neighbor being $(w,3)$, so $(w,3)$ is filled by $(w,2)$. The remaining vertices in row 3 form a path, so they are filled in reverse order from right to left. Recall that all even rows are filled. Therefore, by the same reasoning as row $3$, row $5$ is filled and so on through row $h-2$. Then, $(w,h-1)$ can fill $(w,h)$. Note, $deg(w,h)=2$, with $(w-1,h)$ unfilled. Therefore, $(w,h)$ fills $(w-1,h)$. Since the remaining vertices in row $h$ form a path, they are filled in reverse order from right to left. Thus, the whole graph is filled and $Z(G) \leq {\frac{h+1}{2}}$.
        
        Thus, if $U=\{1\}$ or $U=\{w\}$, then $Z(G) \leq \ceil[\big]{\frac{h}{2}}$.

\textbf{Case 2:} \\
Assume $U=\{i\}$ and $i \neq 1,w$. Let
    \begin{center}
         $S=\{(1,1),(1,2),...,(1,h)\}$.
    \end{center}
    We will show that $S$ is a zero forcing set. All of these vertices have degree 1, so each vertex in the initial set fills their neighbor in column $2$. Therefore, all of column 2 is filled. Since the remaining vertices in each row form a path until column $i$, they are all filled in order from left to right. Thus, vertices in columns $1,2,...,i$ are all filled. The vertex $(i,j)$ for $j=1,...,h$ has one unfilled neighbor $(i+1,j)$. Hence, $(i,j)$ fills $(i+1,j)$ resulting in column $i+1$ begin entirely filled. The remaining vertices are filled in order from left to right across each row since the vertices are in paths. Thus, the whole graph is filled and $Z(G) \leq h$.

\end{proof}

For $G=P_w(U) \sqcap P_h$ with $U=\{i\}$, the minimum rank of the associated matrices of these graphs for some particular factors were found. The minimum rank of the associated matrices provides a lower bound on $Z(G)$. In some cases, this lower bound equals the upper bound on $Z(G)$ found with Theorem \ref{pathsize1}. Thus, in these cases $Z(G)$ exactly is found.

%For example, we found a pattern for $P_2(U) \sqcap P_h$ with $U=\{1\}$ as $h$ increases, that holds with Theorem 5.1 Case 1.

\begin{figure} [H]
\begin{center}
    \begin{tikzpicture}
        \vertexdef
        
    \node[vertex] (1) at (0,1.5){1};
    \node[vertex] (2) at (1,1.5){2};
    \node[vertex] (3) at (0,0.5){3};
    \node[vertex] (4) at (1,0.5){4};
    \node[vertex] (5) at (0,-0.5){5};
    \node[vertex] (6) at (1,-0.5){6};
    \node[vertex] (7) at (0,-1.5){n-3};
    \node[vertex] (8) at (1,-1.5){n-2};
    \node[vertex] (9) at (0,-2.5){n-1};
    \node[vertex] (10) at (1,-2.5){n};
    
    \draw (1) -- (2);
    \draw (1) -- (3);
    \draw (3) -- (4);   
    \draw (3) -- (5); \draw (5) -- (6); \draw (5) -- (0,-0.9);  \draw (7) -- (8); \draw (7) --(9); \draw (9) -- (10);
    \filldraw[black] (0,-0.9) circle (0.5pt); \filldraw[black] (0,-1) circle (0.5pt); \filldraw[black] (0,-1.1) circle (0.5pt) {};
    \filldraw[black] (1,-0.9) circle (0.5pt); \filldraw[black] (1,-1) circle (0.5pt); \filldraw[black] (1,-1.1) circle (0.5pt) {};

    \end{tikzpicture}
    \caption{$G = P_2(U)\sqcap P_h$ with $U=\{1\}$}
    \label{path2}
\end{center}
\end{figure}
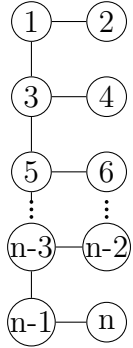

Figure \ref{path2} depicts $G=P_2(U) \sqcap P_h$ with $U=\{1\}$ of an arbitrary height. When $h=3$, Theorem \ref{pathsize1} states that $Z(G) \leq 2$. An associated matrix with minimum rank is the following matrix $A(G)$.

$$
A(G) = \begin{pmatrix}
-1 & 1 & 1 & 0 & 0 & 0 \\
 1 & -1 & 0 & 0 & 0 & 0 \\
 1 & 0 & 0 & 1 & 1 & 0 \\
 0 & 0 & 1 & 0 & 0 & 0 \\
 0 & 0 & -1 & 0 & -1 & 1 \\
 0 & 0 & 0 & 0 & 1 & -1 \\
\end{pmatrix}
$$

Matrix $A(G)$ has rank $4$. From Proposition \ref{minrank}, $2 = 6-4 \leq Z(G)$. Thus, using both Theorem \ref{pathsize1} and Proposition \ref{minrank}, $Z(G)=2$ when $G=P_2(U) \sqcap P_3$ with $U=\{1\}$.

As $h$ increases, we can uniformly change the associated matrix of the graph to get the desired rank to find equality for $Z(G)$. Consider $G' = P_2(U) \sqcap P_4$ with $U = \{1\}$; that is, $h$ has increased by $1$ from $G$.  When $h$ is increased by $1$, a new row is added to the hierarchical product, and two new vertices are added to the graph. These two vertices will be labeled $7$ and $8$ as in Figure \ref{path2}. In addition, there will be two new columns and two new rows in $A(G')$, which will be the $7^{th}$ and $8^{th}$ columns as well as the $7^{th}$ and $8^{th}$ rows. As a result, the zero and nonzero entries of $A(G)$ and $A(G')$ are in the same places. Let $A(G')$ have the same nonzero entries as $A(G)$. In addition, let the nonzero entries in the $7^{th}$ and $8^{th}$ columns and rows of $A(G')$ be $1$. Let the diagonal entries of row $7$ and row $8$ be $0$.  This associated matrix $A(G')$ is shown below. 
\[
A(G') = \begin{pmatrix}
-1 & 1 & 1 & 0 & 0 & 0 & 0 & 0\\
 1 & -1 & 0 & 0 & 0 & 0 & 0 & 0  \\
 1 & 0 & 0 & 1 & 1 & 0 & 0 & 0\\
 0 & 0 & 1 & 0 & 0 & 0 & 0 & 0 \\
 0 & 0 & 1 & 0 & -1 & 1 & 1 & 0 \\
 0 & 0 & 0 & 0 & 1 & -1 & 0 & 0\\
 0 & 0 & 0 & 0 & 1 & 0 & 0 & 1\\
 0 & 0 & 0 & 0 & 0 & 0 & 1 & 0\\
\end{pmatrix}
\]

Matrix $A(G')$ rank $6$. From Proposition \ref{minrank}, we know $2 = 8-6 \leq Z(G')$. By Theorem \ref{pathsize1}, we also know $Z(G') \leq 2$. Thus, $Z(G')=2$ when $G'=P_2(U) \sqcap P_3$ with $U=\{1\}$.

We can generalize this construction to all values of $h$. Consider $G = P_2(U) \sqcap P_{h-1}$ with $U = \{1\}$ and $G' = P_2(U) \sqcap P_{h}$ with $U = \{1\}$. When $h$ is increased by $1$, a new row is added to the hierarchical product. These two new vertices will be labeled $n- 1$ and $n$ as in Figure \ref{path2}. This increase in the number of vertices in the graph results in two new columns and two new rows in $A(G')$, which will be column $(n-1)$ and column $n$ as well as row $(n-1)$ and row $n$. As a result, the zero and nonzero entries of $A(G)$ and $A(G')$ are in the same place, and we will let $A(G')$ have the same nonzero entries as $A(G)$. Also, let the nonzero entries in the new columns and rows of $A(G')$ be $1$. If $h$ is even, let the diagonal entries row $n-1$ and row $n$ be $0$ . If $h$ is odd, let the diagonal entries row $n-1$ and row $n$ be $-1$. This constructed associated matrix $A(G')$ will have minimum rank.

\subsubsection{Root Set of Cardinality 2}
Consider $P_w(U) \sqcap P_h$ with $|U|=2$.
\begin{thm}
Let $G=P_w(U) \sqcap P_h$.
\begin{enumerate}
    \item If $U=\{1,2\}$ or $U=\{w-1,w\}$, then $Z(G) \leq \ceil[\big]{\frac{h}{2}}$.
    \item If $U=\{i,j\}$ with $i \neq 1$, $j \neq w$, and $j \neq i+1$, then $Z(G) \leq h$.
    \end{enumerate}
\end{thm}

\begin{proof}
\textbf{Case 1:}\\
Assume $h$ is even and $U=\{w-1,w\}$. Note, the case in which  $U=\{1,2\}$, is similar since the graphs are isomorphic. Let

    \begin{center}
        $S=\{(1,1), (1,2), (1,4), (1,6),...,(1,h-2)\}.$
    \end{center}
                
 That is, $S$ consists of the vertices in the even rows of column 1 except row $h$ as well as the vertex $(1,1)$.  We will show that $S$ is a zero forcing set. Note that all of these vertices have degree 1 and the remaining vertices in these rows through column $w-1$ are paths, so they are filled in order from left to right across each row. Now, all of the rows that have a vertex in the initial set are filled in columns $1,...,w-1$.
                
Note, $deg(w-1,1)=3$ and its only unfilled neighbor is $(w,1)$. Hence, $(w-1,1)$ fills $(w,1)$. The degree of $(w,1)$ is 2 and $(w,2)$ is the only unfilled neighbor. Hence, $(w,1)$ fills $(w,2)$. Note, $deg(w-1,2)=4$ and $deg(w,2)=3$ and each has one unfilled neighbor. Therefore, $(w-1,3)$ and $(w,3)$ are filled simultaneously by $(w-1,2)$ and $(w,2)$ respectively. Further, $deg(w-1,3)=4$ and it has only one unfilled neighbor, $(w-2,3)$. So, $(w-2,3)$ is filled by $(w-1,3)$. The remainder of vertices in this row are filled in reverse order because they are in a path. Now, rows $1,2,3$ are filled and row $4$ is filled for vertices $(1,4)$ to $(w-1,4)$. Thus, vertex $(w,3)$ which has degree 3, can fill $(w,4)$. Rows $5$ and $6$ are filled similarly to $3$ and $4$ and so on until all the rows through $h-2$ are filled. 
                
Finally, we consider rows $h-1$ and $h$. Vertices $(w-1,h-2)$ and $(w,h-2)$ fill $(w-1,h-1)$ and $(w,h-1)$ respectively. Note, $deg(w-1,h-1)=4$ but it currently has 2 unfilled neighbors and $deg(w,h-1)=3$ and has only one unfilled neighbor. So, $(w,h)$ is filled. The degree of $g(w,h)$ is 2 and it can fill $(w-1,h)$. Now $(w-1,h)$ and $(w-1,h-1)$ both have only one unfilled neighbor. So, $(w-2,h)$ and $(w-2,h-1)$ are filled by the respective vertex. Since the remaining vertices in these two rows are paths, they fill across the row in reverse order. Thus, the whole graph is filled and $Z(G) \leq {\frac{h}{2}}$.
         
Assume $h$ is odd and $U=\{w-1,w\}$. Let      
     \begin{center}
            $S=\{(1,1),(1,2),(1,4), (1,6),...,(1,h-1)\}.$   
     \end{center}
 That is, $S$ consists of the vertices in the even rows of column 1 except row $h$ as well as the vertex $(1,1)$.  We will show that $S$ is a zero forcing set. The first part of the forcing process follows the same reasoning as in case $1$. In which, we demonstrated that the set fills the rows until the second to last even row. In the odd case, the reasoning applies to the last even row, $h-1$. So, rows $1,...,h-1$ are filled.
                
Finally, we consider row $h$. Note, $deg(w-1,h-1)=4$ and $deg(w,h-1)=3$. Both vertices have one unfilled neighbor. So, $(w-1,h)$ and $(w,h)$ are filled respectively. Since the remaining vertices in row $h$ are paths, they fill across in reverse order. Thus, the whole graph is filled and $S$ is a zero forcing set and $Z(G) \leq {\frac{h+1}{2}}$.
                
Thus, if $U=\{1,2\}$ or $U=\{w-1,w\}$, then $Z(G) \leq \ceil[\big]{\frac{h}{2}}$.

\textbf{Case 2:}\\ 
Assume $U=\{i,j\}$ such that $i \neq 1$, $j \neq w$, and $j \neq i+1$ and let
 \begin{center}
         $S=\{(1,1),(1,2),...,(1,h)\}.$
    \end{center}
We will show that $S$ is a zero forcing set. Note that all of these vertices have degree 1 and the remaining vertices in these rows through column $i$ are a path, so they are filled in order from left to right. The vertices in columns $1,...,i$ are filled. The vertex $(i,k)$ for $k=1,...,h$ has one unfilled neighbor $(i+1,k)$.
Hence $(i,k)$ fills $(i+1,k)$. Since the remaining vertices across the rows in columns $i+2,...,j$ are paths, they are filled across in order from left to right. Columns $1,...,j$ are filled. The vertex $(j,k)$ has one unfilled neighbor $(j+1,k)$. Hence $(j,k)$ fills $(j+1,k)$. Since the remaining vertices in columns $j+2,..,w$ are paths, they are filled across in order from left to right. Thus, the whole graph is filled and $Z(G) \leq h$.
            
Now, assume $i \neq 1$, $j \neq w$, and $j=i+1$. We will show that $S$ is a zero forcing set. The first part of the forcing processing follows the same reasoning as the first condition for $i$ and $j$. Hence, the vertices in columns $1,...,i$ are filled. The vertex $(i,k)$ for $k=1,...,h$ has one unfilled neighbor $(j,k)$. Hence $(i,k)$ fills $(j,k)$. The vertex $(j,k)$ has one unfilled neighbor $(j+1,k)$. Hence $(j,k)$ fills $(j+1,k)$. Since the remaining vertices in columns $j+2,..,w$ are paths, they are filled across in order from left to right. Thus, the whole graph is filled and $Z(G) \leq h$.

Next, assume $i \neq 1$ and $j=w$. (Note, suppose $U=\{m,n\}$ such that $m=1$ and $n \neq w$. This case is isomorphic with the first assumption if $|i-j|=|m-n|$.) We will show that $S$ is a zero forcing set. The first part of the forcing processing follows the same reasoning as the first condition for $i$ and $j$. Hence, the vertices in columns $1,...,i$ are filled. The vertex $(i,k)$ for $k=1,...,h$ has one unfilled neighbor $(i+1,k)$. Hence $(i,k)$ fills $(i+1,k)$. Since the remaining vertices across the rows in columns $i+2,...,w$ are paths, so they are filled across in order from left to right. Thus, the whole graph is filled and $Z(G) \leq h$.
            
Assume $i=1$ and $j=w$. We will show that $S$ is a zero forcing set. The vertices $(1,1)$ and $(1,h)$ have a degree 2 and the rest of the vertices in $S$ have degree 3. All of the vertices have 1 unfilled neighbor $(2,k)$ for $k=1,...,h$. Hence, $(1,k)$ fills $(2,k)$ and all of column $2$ is filled. Since the remaining vertices across the rows in columns $3,..,w$ are paths, they are filled across in order from left to right. Columns $1,...,w$ are filled. Thus, the whole graph is filled and $Z(G) \leq h$.

\end{proof}

\subsection{Cycle by Cycle}
Next, we consider the hierarchical product of two cycles. Since each vertex of a cycle has degree $2$, we do not have to consider the separate cases when vertices of degree $1$ are in the root set as is the case with the path by path hierarchical products. Fore example, Figure \ref{cycles} is a hierarchical product of two cycles when the root set consists of vertices $1$ and $2$. The zero forcing number is the same if the root set consists of any two adjacent vertices since the graphs will be isomorphic. Thus, what is important is whether or not vertices in the root set are adjacent. 
    
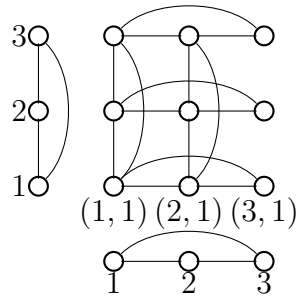
\begin{figure} [H]
\begin{center}
    \begin{tikzpicture}
        \vertexdef
            \filldraw[color=black, fill=white, thick](1,0) circle (0.125) node[anchor=north] {$1$}; %1
         \filldraw[color=black, fill=white, thick](2,0) circle (0.125) node[anchor=north] {$2$};  %2
         \filldraw[color=black, fill=white, thick](3,0) circle (0.125) node[anchor=north] {$3$};%3
         
         \draw (1.1,0.1)..controls (1.5,0.5) and (2.5,0.5)..(2.9,0.1);
         \draw (1.1,0) -- (1.9,0);
         \draw (2.1,0) -- (2.9,0);
        
             \filldraw[color=black, fill=white, thick](0,1) circle (0.125) node[anchor=east] {$1$};%4
             \filldraw[color=black, fill=white, thick](0,2) circle (0.125) node[anchor=east] {$2$};%5
             \filldraw[color=black, fill=white, thick](0,3) circle (0.125) node[anchor=east] {$3$};%6
           
           \draw (1.1,1.1)..controls (1.5,1.5) and (2.5,1.5)..(2.9,1.1);
           \draw (0,1.1) -- (0,1.9);
           \draw (0,2.1) -- (0,2.9);
           
            \filldraw[color=black, fill=white, thick](1,1) circle (0.125) node[anchor=north] {$(1,1)$};%7
            \filldraw[color=black, fill=white, thick](2,1) circle (0.125) node[anchor=north] {$(2,1)$};%8
            \filldraw[color=black, fill=white, thick](3,1) circle (0.125) node[anchor=north] {$(3,1)$};%9
            
            \draw (1.1,2.1)..controls (1.5,2.5) and (2.5,2.5)..(2.9,2.1);
            \draw (1.1,1) -- (1.9,1);
            \draw (2.1,1) -- (2.9,1);
            \draw (1,1.1) -- (1,1.9);
           \draw (1,2.1) -- (1,2.9);
            
             \filldraw[color=black, fill=white, thick](1,2) circle (0.125) {};%10
             \filldraw[color=black, fill=white, thick](2,2) circle (0.125) {};%11
            \filldraw[color=black, fill=white, thick](3,2) circle (0.125) {};%12
            \draw (1.1,3.1)..controls (1.5,3.5) and (2.5,3.5)..(2.9,3.1);
            \draw (1.1,2) -- (1.9,2);
            \draw (2.1,2) -- (2.9,2);
            \draw (2,1.1) -- (2,1.9);
           \draw (2,2.1) -- (2,2.9);
            
          \filldraw[color=black, fill=white, thick](1,3) circle (0.125) {};%13
          \filldraw[color=black, fill=white, thick](2,3) circle (0.125) {};%14
          \filldraw[color=black, fill=white, thick](3,3) circle (0.125) {};%15
           
           \draw (1.1,3) -- (1.9,3);
            \draw (2.1,3) -- (2.9,3);
          
            \draw (0.1,1.1)..controls (0.5,1.5) and (0.5,2.5)..(0.1,2.9);
            \draw (1.1,1.1)..controls(1.5,1.5)and (1.5,2.5)..(1.1,2.9);
            \draw (2.1,1.1)..controls (2.5,1.5) and (2.5,2.5)..(2.1,2.9);

    \end{tikzpicture}
    \caption{$G = C_3(U)\sqcap C_3$ with $U=\{1,2\} $}
    \label{cycles}
\end{center}
\end{figure}

\subsubsection{Root Set of Cardinality 1}
Consider $G=C_w (U) \sqcap C_h$ with $|U| = 1$.

 \begin{thm}
    {Let $G=C_w (U) \sqcap C_h$. If U=\{i\} and $h,w\geq4$, then $Z(G)\leq h+2$}.
    \end{thm}

\begin{proof}Without loss of generality, let $U= \{1\}$. Let
\begin{center}
 $S=$ \{$(1,1),(1,2),(2,1),(2,2),(2,3),...,(2,h)$\}.
\end{center}
That is, $S$ consists of the vertices in column 2 as well as vertices (1,1) and (1,2). Note row 1 is a cycle with two adjacent vertices filled, (1,1) and (2,1). Also, vertex (1,1) cannot fill any vertices in the cycle; however, the vertex (2,1) can fill row 1. Similarly, (2,2) can fill row 2. 
Next, vertex $(1,2)$ has degree 4 and is adjacent to one unfilled vertex, $(1,3)$. So, $(1,2)$ fills $(1,3)$. Row 3 has the same initial conditions as row 2. Thus, row 3 is filled. Now, $(1,3)$ can  fill the vertex $(1,4)$. This process repeats through row h and every vertex is filled. 

Thus, the entire graph is filled and Z(G)$\leq$ $h+2$.
\end{proof}

\subsubsection{Root Set of Cardinality 2}
Consider $G=C_w (U) \sqcap C_h$ with $|U| = 2$.
    \begin{thm}
     Let $G=C_w(U) \sqcap C_h$ and $h,w\geq4$.
     
     \begin{enumerate}
    \item If $U=\{i,j\}$ with $j=i+1$, then $Z(G) \leq h$.
    \item If $U=\{i,j\}$ with $j \neq i+1$, then $Z(G) \leq 2h$.
    \end{enumerate}
    \end{thm}
    
   \begin{proof}
   \textbf{Case 1:} \\
  Without loss of generality, let $U= \{1, 2\}$. Let
\begin{center}
  $S =$ \{$(2,1),(2,3),(2,4),...(2,{h-2}),(2,h),(3,1),(3,h)$\}.
\end{center}
That is, $S$ consist of vertices $(3,1),(3,h)$, and every vertex in the second column except for $(2,2) $ and $(2,{h-1})$. Vertices $(2,h)$ and $(3,h)$ are adjacent. Since vertex $(3,h)$ has a degree two and is adjacent to one unfilled vertex, it can fill $(4,h)$. Row $h$ is a cycle and has two adjacent filled vertices. Then $(3,h)$ and $(4,h)$ can fill the vertices along the row until a vertex in a column associated with a vertex in the root set is filled. Thus, the process continues until $(w,h)$ fills $(1,h)$. The same reasoning can be be applied to fill the vertices in row 1. 
 Vertices $(1,h)$ and $(2,h)$ both have degree four and are adjacent to three filled vertices. Thus, $(1,h)$ and $(2,h)$ can fill $(1,{h-1})$ and $(2,{h-1})$, respectively. By the same reasoning, $(1,1)$ and $(2,1)$ can simultaneously fill $(1,2)$ and $(2,2)$, respectively.
   Vertices $(1,{h-1})$ and $(2,{h-1})$ are adjacent. $(3,{h-1})$ has a degree two with one adjacent vertex filled, thus it can fill $(3,{h-1})$. Again, row ${h-1}$ is a cycle with two filled adjacent vertices. Thus, $(3,{h-1})$ and $(4,{h-1})$ can fill the vertices along the row until a vertex in a column associated with a vertex in the root set is filled. The process continues until $(w,{h-1})$ fills $(1,{h-1})$. The same reasoning can be simultaneously be applied to fill row 1. The previous two steps repeat until row h is filled.
   
   Thus, the entire graph is filled and $Z(G)\leq h$.

\textbf{Case 2:} \\
Without loss of generality, let $U= \{1, 3\}$. Let
\begin{center}
   $S=$ \{$(1,1),(1,2),(1,3),...,(1,h),(2,1),(2,2),(2,3)...,(2,h)$\}.
\end{center}
That is, $S$ consist of every vertex in columns one and two. Every vertex in the column 2 has degree two. Each vertex is adjacent to the filled vertex in column one and the unfilled vertex in column three. Thus, each vertex in column two can fill the adjacent vertex in column three. 
Thus, every vertex in column three is filled. 
Then, each vertex in column three has degree four with three of the vertices filled. Thus, each vertex can fill the vertices in the same row in column four.
The first step repeats until vertices in column $w-1$ fill vertices in column $w$.

Thus, the entire graph is filled and $Z(G)\leq 2h$.

   \end{proof}

\subsubsection{Root Set of Cardinality 3}
Consider $G=C_w (U) \sqcap C_h$ with $|U| = 3$.
    
     \begin{thm}Let $G=C_w(U) \sqcap C_h$ and $h,w\geq4$. 
     {If $U=\{i,i+1,i+2\}$ , then $Z(G)\leq h+2$}.
    \end{thm}
   \begin{proof} Without loss of generality, let $U= \{1, 2, 3\}$. Let 
\begin{center}
   $S=$ \{$(2,1),(2,2),(3,1),(3,2),(3,3),..,(3,h)$\}
\end{center}
That is, $S$ consists of vertices $(3,1), (3,2)$, and all vertices in column 4.
Vertices $(2,1)$ and $(3,1)$ are filled and adjacent. Note row 1 is a cycle with two adjacent vertices, $(2,1)$ and $(3,1)$, filled. Also, vertex $(2,1)$ cannot fill any vertices in the cycle since it has more than one neighbor that is unfilled. However, the vertex (3,1) is in the root set and all the vertices in column 3 are filled. Thus, $(3,1)$ has one unfilled neighbor and can fill row 1 until $(w,1)$ fills $(1,1)$. Thus, row 1 is filled. Similarly, $(3,2)$ can fill row 2. Vertices $(1,2)$ and $(2,2)$ both have degree 4 and are adjacent one unfilled vertex. Thus, they can fill $(1,3)$ and $(2,3)$, respectively. Row 3 has the same initial conditions as row 1. Thus, $(3,3)$ can fill row 3. The process repeats through row $h$ and the entire graph is filled.

Thus, the entire graph is filled and $Z(G)\leq h+2$.
\end{proof}

\subsubsection{Root Set of Cardinality $w-1$}
Consider $G=C_w (U) \sqcap C_h$ with $|U| = w - 1$.
\begin{thm}
     Let $G=C_w(U) \sqcap C_h$ and $h,w\geq4$. {If $U=\{i,i+1,i+2,...,w-1\}$, then $Z(G)\leq 2h-2$}.
    \end{thm}
    \begin{proof}Without loss of generality, let $U= \{2,3,4,...,w\}$. Let
\begin{center}
  $S=$ \{$(2,1),(2,2),...,(2,h),(3,1),(4,1),(5,1),(\ceil{\frac{w}{2}}+1,1),(3,h),(4,h),(5,h),(\ceil{\frac{w}{2}}+1,h)$\}.    \end{center}

  That is, S consist of every vertex in column $2$, vertices $(3,1)$ through $(\ceil{\frac{w}{2}}+1,1)$, and vertices $(3,h)$ through $(\ceil{\frac{w}{2}}+1,h)$.
  Vertices $(3,1)$ through $(\ceil{\frac{w}{2}},1)$ are each adjacent to one unfilled vertex and thus will fill $(3,2)$ through $(\ceil{\frac{w}{2}},2)$. By the same reasoning, $(3,h)$ through $(\ceil{\frac{w}{2}},h)$ can fill $(3,{h-1})$ through $(\ceil{\frac{w}{2}},{h-1})$. filled vertices will continue to fill consecutive rows with one less vertex in the column until $(3,h+3)$ fills $(3,h+2)$ and $(3,h-2)$ fills $(3,h-1)$. Now, vertices $(2,1),(2,2),...,(2,h-1)$ and $(2,h+2),(2,h+3),...,(2,h)$ are each adjacent to one unfilled vertex and can fill vertices $(1,1),(1,2),...,(1,h-1)$ and $(1,h+2),(1,h+3),...,(1,h)$. Vertices $(1,1),(1,2),...,(1,h-1)$ and $(1,h+2),(1,h+3),...,(1,h)$ are each adjacent to one unfilled vertex and can fill vertices $(w,1),(w,2),...,(w,h-1)$ and $(w,h+2),(w,h+3),...,(w,h)$. 
  Vertices $(w,1),(w,2),...,(w,h-2)$ are all adjacent to one unfilled vertex and can fill $(w-1,1),(w-1,2),...,(w-1,h-2)$. Similarly, vertices $(w,h+3),(w,h+4),...,(w,h)$ can fill $(w-1,h+3),(w-1,h+4),...,(w-1,h)$. This process repeats until $(\ceil{\frac{w}{2}}+2,1)$ fills $(\ceil{\frac{w}{2}}+1,1)$ and $(\ceil{\frac{w}{2}}+2,h)$ fills $(\ceil{\frac{w}{2}}+1,h)$.
 Now, row $1$ and row $h$ are filled. The vertices in columns $2$ through $h$ can fill any unfilled vertices in rows $2$ and ${h-1}$. This process repeats until all vertices in columns $2$ through $h$ are filled. 
 Vertices $(2,h)$ and $(2,h+1)$ are now adjacent to one unfilled vertex and can fill $(1,h)$ and $(1,h+1)$. Then entire graph is filled.

Thus, the entire graph is filled and $Z(G)\leq 2h$.
    \end{proof}
    
\subsection{Path by Cycle}

Next, consider the hierarchical product of a path by a cycle. In contrast to cycle by cycle, the placement of the root set causes great variation in the zero forcing number.

\subsubsection{Root Set of Cardinality $m$}
For a path by cycle hierarchical product a unique forcing set occurs for a root set of cardinality $m$ and the vertices in the root set being adjacent on the end.

\begin{thm}
\label{pathendproof}
Let $G=P_w(U) \sqcap C_h$ with $U=\{1,2,...,m\}$ or $U=\{w-m,...,w-1,w\}$ where $0 \leq m < w$.
    \begin{enumerate}
        \item If $h \leq 2m$, then $Z(G) \leq h$.
        \item If $2m < h < 4m$, then $Z(G) \leq 2m$.
        \item If $h \leq 4m$, then $Z(G) \leq \ceil[\big]{\frac{h}{2}}$.
    \end{enumerate}
\end{thm}
\begin{proof}
\textbf{Case 1:}\\
Assume $h<2m$ and $U=\{1,2,...,m\}$. Note the case where $U=\{w-m,...,w-1,w\}$ is isomorphic. Let
\begin{center}
    $S= \{(w,1),(w,2),...,(w,h)\}$.
\end{center}
That is, $S$ consists of all the vertices in column $w$. We will show that $S$ is a zero forcing set. Note that all of the vertices in $S$ have degree $1$ and the remaining vertices in the rows through column $m$ are in a path, so they are filled in reverse order from right to left. The vertices in columns $m,...,w$ are filled. The vertex $(m,k)$ for $k=1,...,h$ has degree $4$ and one unfilled neighbor $((m-1),k)$. Hence, $(m,k)$ fills $((m-1),k)$. This process continues through column $1$. Thus, the whole graph filled and $Z(G) \leq h$.\\

\textbf{Case 2:} \\
Next, without loss of generality, assume $h=4m$ and $U=\{1,2,...,m\}$. Let
\begin{center}
    $S=\{(w,1),(w,2),(w,3),...,(w,2m)\}$.
\end{center}
That is, $S$ consists of the vertices in column $w$ in rows $1,..,2m$.
\begin{enumerate}[(a)]
    \item Note that all of the vertices in $S$ have degree 1 and the remaining vertices in the rows through column $m$ are in a path, so they are filled in reverse order from right to left. The vertices in columns $m,...,w$ and rows $1,...,2m$ are filled.
    \item The vertex $(m,k)$ for $k=1,...,2m$ have degree 4 and all but two vertices have 1 unfilled neighbor. That is, $(m,1)$ and $(m,2m)$ have two unfilled neighbors. Therefore, vertices $(m,k)$ for $k=2,...,2m-1$ can fill their neighbor, $(m-1,k)$. By the same reasoning, columns $m-2,...,1$ have vertices filled. That is, in column $m-i$ for $i=1,...,m-1$, $2m-2i$ vertices are filled. Vertices in rows $(2m-i)+1,...,2m$ and rows $1,...,(1+i-1)$ are not filled for column $i$. When this step of the forcing process is complete, column $1$ has vertices $(1,m)$ and $(1,m+1)$ filled.
    \item The vertices $(1,m)$ and $(1,m+1)$ have degree $3$ and one unfilled neighbor, $(1,m-1)$ and $(1,m+2)$ which are filled by the respective vertex. Now, vertices $(1,m-1)$, $(2,m-1)$, $(1,m+2)$, and $(2,m+2)$ are filled with one unfilled neighbor. Simultaneously, the vertices $(1,m-2)$, $(2,m-2)$, $(1,m+3)$, and $(2,m+3)$ are filled by the respective aforementioned vertex. This reasoning follows for filling the rows $m+1,...,2m$ in order and $m-3,...,1$ in reverse order. As the forcing of these rows progresses, a given row has one more vertex to be filled than the previous row. So, when rows $1$ and $2m$ are filled simultaneously, they each have $m-1$ vertices being filled. Now, all of the vertices in rows $1,...,2m$ are filled.
    \item Note that the vertex $(u,2m)$ for $u$ being a vertex in the root set, has $1$ unfilled neighbor, $(u,2m+1)$. Hence, $(u,2m)$ fills $(u,2m+1)$. Similarly, $(u,1)$ can fill $(u,h)$. (Recall that h=4m.)
    \item Now, vertices $(u,2m+1)$ for $u=1,...,m-1$ have one unfilled neighbor $(u,2m+2)$, so it is filled by $(u,2m+1)$. The same occurs for $(u,h-1)$. This forcing process continues through the rows. Each time vertices are filled in a given row, one less vertex is filled than the number filled in the row that filled them. So, for example vertices $(u,2m+1)$ and $(u,4m)$ for $u=1,..,m$ are filled, and all but $u=m$ has one unfilled neighbor. Hence, $(u,2m+1)$ and $(u,4m)$ for $u=1,...,m-1$ fill $(u,2m+2)$ and $(u,4m-1)$ respectively. Now the vertices $(m-1,2m+2)$ and $(m-1,2m-1)$ are the ones that cannot fill.
    \item This process continues in the rows until vertices $(1,3m)$ and $(1,3m+1)$ are filled and are the only filled vertices in their row. Now in columns $m-i$ for $i=0,...m-2$ there are $2m-(2i+2)$ vertices unfilled. 
    \item Vertices $(1,3m)$ and $(1,3m+1)$ have degree $3$ and one unfilled neighbor, $(2,3m)$ and $(2,3m+1)$ respectively. Hence, $(2,3m)$ and $(2,3m+1)$ are filled by the respective vertex. Similarly, the unfilled vertices in columns $m-(m-2),...,m$ are filled across the rows from left to right. Now, all of the vertices in columns $1,...,m$ are filled. Since the remaining unfilled vertices are in a path in rows $2m+1,...,4m$ they are filled across in order from left to right. Thus, the whole graph is filled and $Z(G) \leq 2m$.
    Note here that $S$, which has $2m$ vertices, when $h=4m$ causes all of column $1$ to be filled. That is, when $S$ has 2m vertices, it can fill up to $4m$ vertices in column $1$. So, for any height between $2m$ and $4m$, $S$ will fill all of the first column. Once all of column $1$ is filled, the rest of the graph can be filled. So any height from $2m,...,4m$ can be filled by $S$. Thus, for $2m \leq h < 4m$, $Z(G) \leq 2m$.\\
 \end{enumerate}
\textbf{Case 3:}\\
Next, suppose $h=4m+2i$ ,that is, $h$ is even and $U=\{1,2,...,m\}$. Let
\begin{center}
    $S=\{(w,1), (w,2),...,(w,2m),(w,2m+2),(w,2m+4),...,(w,h-2m)\}$.
\end{center}
That is, $S$ is vertices in column $w$ in rows $1,...,2m$ and all even rows up to $h-2m$. Note that this is exactly $i$ extra even rows when $h=4m+2i$, so there are $2m+i$ vertices in $S$.

\begin{enumerate}[(a)]
    \item The forcing process begins similarly to Case 2. That is, the forcing process follows through step (e) but when we reach step (f), since $h >4m$, the filling process will not result in all of column $1$ to be filled. Again, vertex $(1,3m)$ will be filled and the only one in its row. %but instead of $(1,3m+1)$, vertex $(1, 3m+(j+1))$ for $j$ is positive. In other words, %
    However, vertices $(1,3m+1),(1,3m+2),...,(1,h-m)$ are unfilled.

    \item Now, consider how the vertices $(w, 2m+2), (w,2m+4), (w, 2m+6),…,(w,h-8), (w,h-6)$ in $S$ force. Note that all of the vertices in $S$ have degree $1$ and the remaining vertices in the rows through column $m$ are in a path, so they are filled in reverse order from right to left. The vertices in columns $m,...,w$  and rows $2m+2, 2m+4, 2m+6...,h-8, h-6$ are filled. 

    \item Vertices $(m,2m+2)$ and $(m-1, 2m+2)$ are filled, have degree $4$, and one unfilled neighbor, $(m, 2m+3)$ and $(m-1, 2m+3)$, respectively. Thus, the latter vertices are filled by the former vertices respectively. Now all vertices corresponding to $(u,j)$ for $u$ in the root set and $1 \leq j\leq 2m+3$ are filled. In row $2m+4$, vertices $(1,2m+4)$ through $(m-2,2m+4)$ are then filled by the vertices in row $2m+3$. Since row $2m+4$ is one of the rows with $(w,2m+4)$ in $S$, vertices $(m,2m+4)$ through $(w,2m+4)$ are filled. This process continues filling all vertices $(u,j)$ in rows $1 \leq j\leq h-2m+1$.

    \item Then, in rows $h-2m+2$ through $h-m$ are filled by the row below with one less vertex filled then the number of filled vertices in the row below. Similarly, rows $h-1$ down to row $h-m+1$ are filled by the row above with one less vertex filled than the number of filled vertices in the row above.

    \item Now, vertices $(1,h-m)$ and $(1,h-m+1)$ are the only filled vertices in their respective rows. Each of these vertices has degree 3 with one unfilled vertex. Thus, vertices $(2, h-m)$ and $(2,h-m+1)$ are filled respectively by $(1,h-m)$ and $(1,h-m+1)$. The first two columns of $G$ are now entirely filled and can continue to fill the remaining vertices until all the vertices in the columns corresponding to the root set are filled. 
    
    \item The remaining unfilled vertices are those in the odd rows that do not have vertices in $S$. Since the vertices in column $m$ for each of these rows is filled and the remaining vertices form a path, the vertices in these rows are successively filled from left to right.

\end{enumerate}

Thus, the whole graph is filled and $Z(G) \leq \frac{h}{2}$.

The case in which $h=4m+(2i-1)$, $i=1,2,...$ is similar. That is, $h$ is odd and without loss of generality let $U=\{1,2,...,m\}$. Let
\begin{center}
    $S=\{(w,1), (w,2),...,(w,2m),(w,2m+2),(w,2m+4),...,(w,h-2m+1)\}$.
\end{center}

That is, $S$ is vertices in column $w$ in rows $1,...,2m$ and all even rows up to $h-2m+1$. Note that this is exactly $i$ extra even rows when $h=4m+2i$, so there are $2m+i$ vertices in $S$.

The forcing process begins similarly to when $h$ is even. That is, steps (a) through (d) are the same. When step (e) is reached, only vertex $(1,h-m)$ is filled in its respective row. It has degree 3 with one unfilled vertex. Thus, vertex $(2,h-m)$ is filled by $(1,h-m)$. The first two columns of $G$ is now entirely filled and can continue to fill the remaining vertices until all the vertices in the columns corresponding to the root set are filled. Step (f) from the even case applies here. Thus, the whole graph is filled and $Z(G) \leq \frac{h+1}{2}$.

Thus, if $h > 4m$, then $Z(G)\leq \ceil[\big]{\frac{h}{2}}$

 \end{proof}

\subsubsection{Root Set of Cardinality 1}

Next we consider path by cycle hierarchical products with one vertex in the root set are considered. Note that the cases in which $U=\{1\}$ or $\{w\}$ are isomorphic.

\begin{thm} Let $G=P_w(U) \sqcap C_h$.
\begin{enumerate}
    \item If $U=\{1\}$ or $U=\{w\}$, then $Z(G) \leq \ceil[\big]{\frac{h}{2}}$.
    \item If $U=\{i\}$ and $i \neq 1,w$, then $Z(G) \leq h$.
\end{enumerate}

\end{thm}

\begin{proof}
\textbf{Case 1:} \\
Case $1$ is proven by Theorem \ref{pathendproof}.

\textbf{Case 2:} \\
Assume $U=\{i\}$ and $i \neq 1,w$. Let
\begin{center}
  $S=$ \{$(1,1), (1,2), ..., (1,h)$\}
\end{center}

We will show that $S$ is a zero forcing set. All of the vertices in $S$ have degree 1, so they can fill their neighbor in column 2. Therefore, all of column 2 is filled. Since the remaining vertices in each row are in a path until column $i$, they are all filled in order from left to right. Thus, vertices in columns $1,2,..,i$ are all filled. The vertex $(i,j)$ for $j=1,...,h$ has degree 4 and one unfilled neighbor, $(i+1,j)$. Hence, $(i,j)$ fills $(i+1,j)$ resulting in column $i+1$ begin entirely filled. The remaining vertices are filled in order from left to right across each row since they are in a path. Thus, the whole graph is filled and $Z(G) \leq h$. 
\end{proof}

\subsubsection{Root Set of Cardinality 2}
We now consider $P_w(U) \sqcap C_h$ with $|U|=2$

\begin{thm} Let $G=P_w (U) \sqcap C_h$.
If $U=\{i,j\}$, then $Z(G) \leq h$.
\end{thm}

\begin{proof} Let
    \begin{center}
         $S=$ \{$(1,1), (1,2), (1,3),..., (1,h)$\}.
    \end{center}
 
Assume $U=\{i,j\}$ such that $i \neq 1$, $j \neq w$, and $j \neq i+1$. We will show that $S$ is a zero forcing set. Note that all of the vertices in $S$ have degree 1 and the remaining vertices in these rows through column $i$ are in a path, so they are filled in order from left to right. Now, the vertices in columns $1,...,i$ are filled. The vertex $(i,k)$ for $k=1,...,h$ has degree 4 and one unfilled neighbor, $(i+1,k)$.
Hence $(i,k)$ fills $(i+1,k)$. Since the remaining vertices across the rows in columns $i+2,...,j$ are paths, they are filled in order from left to right. Columns $1,...,j$ are filled. The vertex $(j,k)$ has degree 4 and one unfilled neighbor $(j+1,k)$. Hence $(j,k)$ fills $(j+1,k)$. Since the remaining vertices in columns $j+2,..,w$ are paths, they are filled in order from left to right across the row. Thus, the whole graph is filled and $Z(G) \leq h$.

 Now, assume $i \neq 1$, $j \neq w$, and $j=i+1$. We will show that $S$ is a zero forcing set. The first part of the forcing processing follows the same reasoning as the first condition for $i$ and $j$. Hence, the vertices in columns $1,...,i$ are filled. The vertex $(i,k)$ for $k=1,...,h$ has degree 4 and one unfilled neighbor $(j,k)$. Hence $(i,k)$ fills $(j,k)$. The vertex $(j,k)$ has degree 4 and one unfilled neighbor $(j+1,k)$. Hence $(j,k)$ fills $(j+1,k)$. Since the remaining vertices in columns $j+2,..,w$ are paths, they are filled in order from left to right across the row. Thus, the whole graph is filled and $Z(G) \leq h$.

Next, assume $i \neq 1$ and $j=w$. (Note: suppose $U=\{m,n\}$ such that $m=1$ and $n \neq w$. This case is isomorphic with the first assumption if $|i-j|=|m-n|$.) We will show that $S$ is a zero forcing set. The first part of the forcing processing follows the same reasoning as the first condition for $i$ and $j$. Hence, the vertices in columns $1,...,i$ are filled. The vertex $(i,k)$ for $k=1,...,h$ has degree 4 and one unfilled neighbor $(i+1,k)$. Hence $(i,k)$ fills $(i+1,k)$. Since the remaining vertices across the rows in columns $i+2,...,w$ are paths, so they are filled in order across the rows from left to right. Thus, the whole graph is filled and $Z(G) \leq h$.
            
Assume $i=1$ and $j=w$. We will show that $S$ is a zero forcing set. The vertices in $S$ have degree 4 and have 1 unfilled neighbor, $(2,k)$ for $k=1,...,h$. Hence, $(1,k)$ fills $(2,k)$ and all of column $2$ is filled. Since the remaining vertices across the rows in columns $3,..,w$ are paths, they are filled across in order from left to right. Columns $1,...,w$ are filled. Thus, the whole graph is filled and $Z(G) \leq h$.

\end{proof}

\subsection{Complete Graph by Complete Graph}

Another special product to consider is the hierarchical product of two complete graphs. This case is particularly interesting because every vertex is adjacent to all vertices in its row as well as all vertices in its column if the vertex is in the root set. Theorem \ref{completebound} gives an upper bound on the zero forcing number for any hierarchical product of a complete graphs by a complete graphs when not all vertices in are in root set.

\begin{thm}
\label{completebound}
Let $G = K_w(U)  \sqcap K_h$ and $h,w\geq 4$. If $|U| = r \neq w$, then $Z(G) \leq wh - (h + r)$.
\end{thm}
\begin{proof} Without loss of generality, let $U = \{i, i + 1 , i + 2,..., {w-1}\}$ for $1 \leq i \leq w - 1$. Thus, we can assume $w \notin U$. Let
    \begin{center}
    $S = V(G) - \{(w-1,1),...,(w-1,h),(w,h), (i, h), (i + 1, h), \ldots, (w - 2, h)\}$
        \end{center}
\noindent
That is, a minimal zero forcing set consists of every vertex in the hierarchical product except for all the vertices in one arbitrary column associated with a vertex in the root set, one vertex in arbitrary column associated with a vertex not in the root set in an arbitrary row, and all vertices in columns associated with vertices in the root set in an arbitrary row. In this example, vertices in column $(w-1)$, vertex $(w,h)$, and vertices $(i, h), (i + 1, h), \ldots, (w - 2, h)$ are chosen to be excluded from the zero forcing set. Note, $|S| = wh - (h + 1 + (r - 1)) = wh - (h + r)$. 

%Figure \ref{completebycompletegraph} demonstrates this set $S$. %The figure below helps demonstrate this more complex set S.

Since $w \notin U$, $(w,1)$ is not adjacent to any vertex in the column $w$. Since $(w, 1)$ is in row $1$, $(w, 1)$ has one unfilled neighbor, $({w-1},1)$. Therefore, $(w,1)$ can fill $({w-1},1)$. The same reasoning can be simultaneously applied to $(w,2),(w,3),(w,4),...,(w, h-2)$, and $(w,{h-1})$, so that they may fill $({w-1},2),({w-1},3),$  $({w-1},4),..., (w - 1, h-2)$ and $({w-1},{h-1})$.

Next, each column $i$ for $1 \leq i \leq w - 1$ has all but one vertex, $(i,h)$, filled. Every vertex in rows one through ${h-1}$ is filled. Thus, any filled vertex in column $i$ has one unfilled neighbor. Therefore, vertex $(i,1)$ can fill vertex $(i,h)$ for $1 \leq i \leq w - 1$.

The last unfilled vertex in the graph is $(w,h)$. Every vertex in row $h$ is filled.  Therefore, $(1,h)$ can fill $(w,h)$.
Thus, the entire graph is filled and $Z(G) \leq wh - (h + r)$.

\end{proof}

The propagation time for any complete graph by complete graph using the set $S$ in Theorem \ref{completebound} is three steps. 

\section{General Product Bound}
\label{generalboundsec}

In Section \ref{specialproducts}, we gave bounds on the zero forcing number of the hierarchical product when the factors were either a path, a cycle, or a complete graph for various root sets. However, we can now consider the zero forcing number of the hierarchical product for any two arbitrary graphs.

\begin{thm}
\label{generalbound}
If $G=W(U) \sqcap H$ where $W$ and $H$ are arbitrary graphs, then \[Z(G) \leq \min  \{Z(W)|V(H)|,  Z(H)|U|+(|V(W)|-|U|)|V(H)|\}.\]
\end{thm}
\begin{proof}
First, we will show $Z(G) \leq Z(W)|V(H)|.$ Let $S_1$ be the set consisting of a minimal zero forcing set of the graph $W$ placed in every row. Note, $|S_1| = Z(W)|V(H)|$. If a vertex in $S_1$ is also in the root set, then the vertex can fill vertices in its row since every vertex in the column is already filled. If a vertex in $S_1$ is not in the root set, then the vertex can fill vertices in its row since the vertex is not adjacent any vertex in its column. Thus, every filled vertex can fill vertices in its row, and $S_1$ will fill every row simultaneously. Thus, the entire graph is filled and $Z(G)$ $\leq$\ $Z(W)|V(H)|.$

Next, we will show $Z(G) \leq$ $Z(H)|U| + (|V(W)|-|U|)|V(H)|$. Let $S_2$ be the set consisting of a minimal zero forcing set of graph $H$ placed in every column associated with a vertex in the root set as well as all the vertices in each column associated with a vertex not in the root set.  Note, $|S_2| = Z(H)|U| + (|V(W)|-|U|)|V(H)|$. Note, the only unfilled vertices of $G$ are in columns associated with a vertex in the root set. In addition, for every vertex in $S_2$, its only unfilled neighbors are in its column. Hence, since $S_2$ consists a minimal zero forcing set of graph $H$ placed in every column associated with a vertex in the root set, $S_2$ will fill every column simultaneously. Thus, the entire graph is filled and $Z(G)$ $\leq$  $Z(H)|U| + (|V(W)|-|U|)|V(H)|$.

Therefore, $Z(G) \leq \min  \{Z(W)|V(H)|, Z(H)|U|+(|V(W)|-|U|)|V(H)|\}.$

\end{proof}

Note, when $U = V(W)$, then $G = W(U) \sqcap H = W \Box H$ and Theorem \ref{generalbound} states $Z(G) \leq \min\{Z(W)|V(H)|, Z(H)|V(G)|\}$, which was first proven in \cite{AIM}. 

\section{Future Work}
In this paper, we found bounds on the zero forcing number of generalized hierarchical products of graphs. Multiple upper bounds were found for various hierarchical products. Equality was established for a few cases by using associated matrices to find a lower bound, but most of the products considered only have an upper bound. Going forward, associated matrices should be rigorously studied to establish lower bounds. Further, there are a number of other hierarchical products to consider. Additionally, a study of propagation time would be beneficial.
\newpage

\appendix

\section{Upper Bounds on the Zero Forcing Number}

This section contains tables that summarize the results given on the upper bound on the zero forcing number of hierarchical products. The tables are organized by product type.

\subsection{$G = P_w(U) \sqcap P_h$}
\begin{tabular}{ |c||c|c|} 
 \hline
 Root Set, $U$ & $Z(G)$ & $pt(G,S)$\\ \hline \hline
 $U=\{i\}$ & $\ceil[\big]{\frac{h}{2}}$ &  If $h$ is even:  \\
 $i=1$ or $w$  & &  $2(w-1) +\frac{h-2}{2}$\\
  & & If $h$ is odd:\\
  & & $2(w-1)+\frac{h-1}{2}$\\ \hline

 $U=\{i\}$ &  $h$ &  $w-1$ \\ 
 $i \neq 1,w$ & & \\\hline
 
$U=\{1,2\}$ &  $\ceil[\big]{\frac{h}{2}}$ & If $h$ is even:\\
or &  &  $h+1+2(w-2)$\\
$U=\{w-1,w\}$ & & \\
 & & If $h$ is odd:\\ 
 & &$h+2(w-2)$ \\\hline

 $|U|=2$  & $h$ & $w-1$ \\ 
 $U \neq \{1,2\}$ &  &  \\ 
 or $\{w-1,w\}$ & & \\\hline
 
\end{tabular}

\subsection{$G=C_w(U) \sqcap C_h$}
\begin{tabular}{ |c||c|c|} 
 \hline
 Root Set, $U$ & $Z(G)$ & $pt(G,S)$ \\ \hline \hline
 
 $U=i$  & $h+2$ &  $h$ is even: \\ 
  $i$ is any vertex &   & $\frac{h(w-2)}{2}+\frac{h-2}{2}$ \\
  & & $h$ is odd\\
  & & $\frac{h(w-2)}{2} +\frac{h-2}{2}$\\ 
  & &$+(w-2)+1$\\ \hline
 
 $U=\{i,j\}$ & $h$ & $h$ is even:\\
 $i$ is any vertex; &  & $\frac{h(w-2)}{2}+\frac{h-2}{2}$\\
 $j=i+1$ & & $h$ is odd: \\
 & & $\frac{h(w-2)}{2}+\frac{h-2}{2}$ \\
 & & $+(w-2)+1$\\ \hline
 
 $U=\{i,j\}$ &  $2h$ & $w-2$ \\ 
 $i$ is any vertex; &  & \\
 $j \neq i+1$ & &\\ \hline
 
 $U=\{i,j,k\}$ & $h+2$ & $h$ is even:\\
 $i$ is any vertex; & & $\frac{h(w-2)}{2}+\frac{h-2}{2}$ \\
 $j=i+1$ & & $h$ is odd: \\
 $k=i+2$ & & $\frac{h(w-2)}{2}+\frac{h-2}{2}$ \\
  & & $+(w-2)+1$ \\\hline
 
 $U=$ & $2h-2$ & $\frac{(h-2)}{2}+{w-2}$\\
 $\{i,j,k,...,w-1\}$ & & \\ \
  & &\\ \hline

\end{tabular}

\subsection{$G=P_w(U) \sqcap C_h$}

\begin{tabular}{ |c||c|c|} 
 \hline
 Root Set, $U$ & $Z(G)$ & $pt(G,S)$\\ \hline \hline
 
&  & $h$ is even:\\
$U=\{1$\} or & $\ceil[\big]{\frac{h}{2}}$ & $2(w-1)+\frac{h-2}{2}$\\
$\{w\}$ & &\\
 & &\\\hline

$U=\{i\}$ & $h$ & $w-1$\\
$i \neq 1,w$ & & \\ \hline
 
 $U=\{i,j\}$ & $h$ & $w-1$\\
 $i \neq 1,w-1$ & &\\
 $j \neq 2, w$ & & \\\hline
 
\end{tabular}

\subsection{$G=K_w (U) \sqcap K_h$}

\begin{tabular}{ |c||c|c|} 
 \hline
 Root Set, $U$ & $Z(G)$ &  $pt(G,S)$ \\ \hline \hline
 $U=i$ & $(w-1)*(h-1) $ &  $3$ \\ 
 $i \neq w$ &  $+(w-r)-1$ &\\ \hline
\end{tabular}

\end{document}